\documentclass{amsart}
\usepackage{graphicx}
\usepackage{amsfonts,amsmath}
\begin{document}

 \newtheorem{thm}{Theorem}[section]
 \newtheorem{cor}[thm]{Corollary}
 \newtheorem{lem}[thm]{Lemma}{\rm}
 \newtheorem{prop}[thm]{Proposition}

 \newtheorem{defn}[thm]{Definition}{\rm}
 \newtheorem{assumption}[thm]{Assumption}
 \newtheorem{rem}[thm]{Remark}
 \newtheorem{ex}{Example}
\numberwithin{equation}{section}
\def\e{{\rm e}}
\def\x{\mathbf{x}}
\def\by{\mathbf{y}}
\def\bz{\mathbf{z}}
\def\F{\mathcal{F}}
\def\R{\mathbb{R}}
\def\T{\mathbf{T}}
\def\N{\mathbb{N}}
\def\K{\mathbf{K}}
\def\Q{\mathbf{Q}}
\def\M{\mathbf{M}}
\def\O{\mathbf{O}}
\def\C{\mathbb{C}}
\def\P{\mathbf{P}}
\def\Z{\mathbb{Z}}
\def\H{\mathcal{H}}
\def\A{\mathbf{A}}
\def\V{\mathbf{V}}
\def\AA{\overline{\mathbf{A}}}
\def\B{\mathbf{B}}
\def\c{\mathbf{C}}
\def\L{\mathbf{L}}
\def\bS{\mathbf{S}}
\def\H{\mathcal{H}}
\def\I{\mathbf{I}}
\def\Y{\mathbf{Y}}
\def\X{\mathbf{X}}
\def\G{\mathbf{G}}
\def\f{\mathbf{f}}
\def\z{\mathbf{z}}
\def\y{\mathbf{y}}
\def\d{\hat{d}}
\def\bx{\mathbf{x}}
\def\y{\mathbf{y}}
\def\v{\mathbf{v}}
\def\g{\mathbf{g}}
\def\w{\mathbf{w}}
\def\b{\mathbf{b}}
\def\a{\mathbf{a}}
\def\u{\mathbf{u}}
\def\s{\mathcal{S}}
\def\cc{\mathcal{C}}
\def\co{{\rm co}\,}
\def\tg{\tilde{f}}
\def\tx{\tilde{\x}}
\def\supmu{{\rm supp}\,\mu}
\def\supphi{{\rm supp}\,\varphi}
\def\l{{\rm !}}

\title[nonnegativity on closed sets and optimization]{
A new look at nonnegativity on closed sets
and polynomial optimization}
\author{Jean B. Lasserre}
\address{LAAS-CNRS and Institute of Mathematics\\
University of Toulouse\\
LAAS, 7 avenue du Colonel Roche\\
31077 Toulouse C\'edex 4,France}
\email{lasserre@laas.fr}
\date{}

\begin{abstract}
We first show that
a continuous function $f$ 
is nonnegative on a closed
set $\K\subseteq\R^n$ if and only if
(countably many) moment matrices of some signed
measure $d\nu =fd\mu$ with $\supmu=\K$, are all positive semidefinite (if $\K$ is compact $\mu$ is an arbitrary finite Borel measure with $\supmu=\K$). 
In particular, we obtain a convergent explicit hierarchy
of semidefinite (outer) approximations with {\it no} lifting, of the cone of nonnegative polynomials of degree at most $d$.
Wen used in polynomial optimization on certain simple closed sets $\K$ (like e.g., the whole space $\R^n$, the positive orthant, a box, a simplex, or the vertices of the hypercube),
it provides a nonincreasing sequence of upper bounds
which converges to the global minimum by solving a hierarchy of semidefinite programs with only one variable
(in fact, a generalized eigenvalue problem). In the compact case, this convergent
sequence of upper bounds complements the convergent sequence of lower bounds obtained by solving a 
hierarchy of semidefinite relaxations as in e.g. \cite{lasserresiopt}.
\end{abstract}

\keywords{closed sets; nonnegative functions; 
nonnegative polynomials;
semidefinite approximations; moments}

\subjclass{90C25 28C15}

\maketitle

\section{Introduction}

This paper is concerned with a concrete characterization
of continuous functions that are nonnegative on a closed
set $\K\subseteq\R^n$ and its application for optimization purposes.  By concrete we mean a systematic
procedure, e.g. a numerical test that can be implemented
by an algorithm, at least in some interesting cases. For polynomials,
Stengle's Nichtnegativstellensatz \cite{stengle}
provides a certificate
of nonnegativity (or absence of nonnegativity) 
on a semi-algebraic set. Moreover, in principle, this certificate 
can be obtained by solving a single semidefinite program
(although the size of this semidefinite program 
is far beyond the capabilities of today's computers).
Similarly, for compact basic semi-algebraic sets, Schm\"udgen's and Putinar's Positivstellens\"atze 
\cite{schmudgen,putinar} provide certificates 
of strict positivity that can be obtained by solving finitely many semidefinite programs (of increasing size).
Extensions of those certificates to some algebras of non-polynomial functions have been recently proposed in 
Lasserre and Putinar \cite{lasput} and in Marshall and Netzer \cite{netzer}. However, and to the best of our knowledge, there is still no hierarchy of explicit (outer or inner) semidefinite approximations 
(with or without lifting) of the cone of polynomials nonnegative on a closed set $\K$, except if $\K$ is compact and basic semi-algebraic (in which case outer approximations exist). Another exception 
is the convex cone of quadratic forms nonnegative on $\K=\R^n_+$
for which inner and outer approximations are available; see e.g. Anstreicher and Burer \cite{copo1}, and D\"ur \cite{copo2}.
\\

\noindent
{\bf Contribution:} In this paper, we present 
a different approach based on a new (at least
to the best of our knowledge) and simple characterization of continuous functions that are nonnegative on a  closed set 
$\K\subseteq\R^n$. This characterization involves
a {\it single} (but known) measure $\mu$ with $\supmu=\K$, and sums of squares of polynomials. 
Namely, our contribution is twofold:

(a) We first show that a continuous function $f$ is nonnegative on a closed set $\K\subseteq\R^n$
if and only if $\int h^2fd\mu$ is nonnegative for all
polynomials $h\in\R[\x]$, where $\mu$
is a finite Borel measure\footnote{A finite Borel measure $\mu$ on $\R^n$ is a nonnegative set function defined on the Borel 
$\sigma$-algebra of $\R^n$ (i.e., the $\sigma$-algebra generated by the open sets), such that $\mu(\emptyset)=0$, 
$\mu(\R^n)<\infty$, and $\mu(\bigcup_{i=1}^\infty E_i)=\sum_{i=1}^\infty \mu(E_i)$ for any collection of disjoint measurable sets $E_i$.
Its support (denoted $\supmu$) is the smallest closed set 
$\K$ such that $\mu(\R^n\setminus\K)=0$; see e.g. Royden \cite{royden}.} with $\supmu=\K$. 
The measure $\mu$ is arbitrary if $\K$ is compact.
If $\K$ is not compact  then one may choose for $\mu$ the finite Borel measure:

- $d\mu =\exp(-\sum_i \vert x_i\vert)d\varphi$ if $f$ is a polynomial, and

- $d\mu =(1+f^2)^{-1}\exp(-\sum_i \vert x_i\vert)d\varphi$, if $f$ is not a polynomial,

\noindent
where $\varphi$ is any finite Borel measure with support exactly $\K$. But many other choices are possible.

Equivalently, $f$ is nonnegative on $\K$ if and only if
every element of the countable family $\mathcal{T}$ of moment matrices 
associated with the signed Borel 
measure $fd\mu$, is positive semidefinite. 
The absence of nonnegativity on $\K$ can be {\it certified} by exhibiting a polynomial $h\in\R[\x]$ such that
$\int h^2fd\mu<0$, or equivalently,
when some moment matrix in the family $\mathcal{T}$
is not positive semidefinite. And so, interestingly, as for nonnegativity or positivity, our certificate
for absence of nonnegativity is also in terms of sums of squares. When $f$ is a polynomial, these moment matrices are easily
obtained from the moments of $\mu$ and this criterion for absence of nonnegativity complements Stengle's Nichtnegativstellensatz \cite{stengle} (which provides a certificate
of nonnegativity on a semi-algebraic set $\K$) or Schm\"udgen and Putinar's Positivstellens\"atze \cite{schmudgen,putinar} (for 
certificates of strict positivity on compact basic semi-algebraic sets).  At last but not least, we obtain a convergent {\it explicit} hierarchy
of semidefinite (outer) approximations with {\it no} lifting, of the cone $\cc_d$ of nonnegative polynomials of degree at most $2d$. That is, we obtain a nested sequence $\cc^0_d\supset\cdots\cc^k_d\supset\cdots\supset\cc_d$ such that
each $\cc^k_d$ is a spectrahedron defined solely in terms of 
the vector of coefficients of the polynomial, with no additional variable (i.e., no projection is needed). Similar explicit hierarchies can be obtained for the cone of polynomials nonnegative on a  closed set $\K$ (neither
necessarily basic semi-algebraic nor compact), provided that all moments of an appropriate measure $\mu$ (with support exactly $\K$) can be obtained. To the best of our knowledge, this is first result of this kind.

(b) As a potential application, we  consider the problem of computing the {\it global} minimum
$f^*$ of  a polynomial $f$ on a closed set $\K$, a notoriously difficult problem.
In nonlinear programming, a sequence of upper bounds on $f^*$ is usually obtained from a sequence of feasible points 
$(\x_k)\subset\K$, e.g., via some (local) minimization algorithm. But it is important to emphasize that for non convex problems, 
providing a sequence of upper bounds $(f(\x_k))$, $k\in\N$, that converges to $f^*$
is in general impossible, unless one computes points on a grid whose mesh size tends to zero.

We consider the case where $\K\subseteq\R^n$ is a closed set
for which one may compute
all moments of a measure $\mu$ with $\supmu=\K$.
Typical examples of such sets are e.g. $\K=\R^n$ or $\K=\R^n_+$ in the non compact case and
a box, a ball, an ellipsoid, a simplex, or the vertices of an hypercube (or hyper rectangle) in the compact case. 
We then provide a hierarchy of semidefinite programs (with only one variable!) whose optimal values form a monotone nonincreasing sequence of {\it upper} bounds
which converges to the global minimum $f^*$. In fact, each semidefinite program is a very specific one as it
reduces to solving a {\it generalized eigenvalue} problem for a pair of real symmetric matrices..
(Therefore, for efficiency one may use specialized software packages instead of a SDP solver.)
However, the convergence to $f^*$ is in general only asymptotic and not finite (except when $\K$ is a discrete set
in which case finite convergence takes place).
This is in contrast with the hierarchy of semidefinite relaxations defined in Lasserre \cite{lasserresiopt,lasserrebook} 
which provide  a nondecreasing sequence of {\it lower} bounds
that also converges to $f^*$, and very often in finitely many steps.
Hence, for compact basic semi-algebraic sets these two convergent hierarchies of upper and lower bounds complement each other and permit to locate 
the global minimum $f^*$ in smaller and smaller intervals.

Notice that convergence of the hierarchy of convex relaxations in \cite{lasserresiopt} is guaranteed only for compact basic semi-algebraic sets, whereas for the new hierarchy of upper bounds, the only requirement on $\K$ is to know all moments of a measure $\mu$ with $\supmu=\K$. On the other hand, in general
computing such moments is possible only for relatively simple (but not necessarily compact nor semi-algebraic) sets.

At last but not least, 
the nonincreasing sequence of upper bounds converges to $f^*$ even if $f^*$ is not attained, 
which when $\K=\R^n$, could provide
an alternative and/or a complement to the hierarchy of convex relaxations 
provided in Schweighofer \cite{markus} (based on gradient tentacles)
and in H\`a and Pham \cite{ha} (based on the truncated tangency variety), which both provide again a monotone sequence of lower bounds.

Finally, we also give a very simple interpretation
of the hierarchy of dual semidefinite programs, which 
provides some information on the location of global minimizers.

\section{Notation, definitions and preliminary results}
A  Borel measure on $\R^n$ is understood as a {\it positive} Borel measure, i.e., a 
nonnegative set function $\mu$ on the Borel $\sigma$-algebra $\mathcal{B}$ (i.e., the $\sigma$-algebra generated by
the open sets of $\R^n$) such that $\mu(\emptyset)=0$, and
with the countably additive property
\[\mu\left(\bigcup_{i=1}^\infty E_i\right)\,=\,\sum_{i=1}^\infty \mu(E_i),\]
for any collection of disjoint measurable sets $(E_i)\subset\mathcal{B}$; see e.g. Royden \cite[pp. 253--254]{royden}.

Let $\R[\x]$ be the ring of polynomials in the variables
$\x=(x_1,\ldots,x_n)$, and $\Sigma[\x]\subset\R[\x]$ its subset of polynomials that are sums of squares (s.o.s.).
Denote by $\R[\x]_d\subset\R[\x]$ the vector space of
polynomials of degree at most $d$, which forms a vector space of dimension $s(d)={n+d\choose d}$, with e.g.,
the usual canonical basis $(\x^\alpha)$ of monomials.
Also, denote by $\Sigma[\x]_d\subset\Sigma[\x]$
the convex cone of s.o.s. polynomials of degree at most $2d$. If $f\in\R[\x]_d$, write
$f(\x)=\sum_{\alpha\in\N^n}f_\alpha \x^\alpha$ in
the canonical basis and
denote by $\f=(f_\alpha)\in\R^{s(d)}$ its vector of coefficients. 
Let $\mathcal{S}_n$ denotes the vector space of $p\times p$ real symmetric matrices. For a matrix
$\A\in\mathcal{S}_p$ the notation $\A\succeq0$ (resp. $\A\succ0$) stands for $\A$ is positive semidefinite
(resp. definite).

\subsection*{Moment matrix} With $\y=(y_\alpha)$ being a sequence indexed in the canonical basis
$(\x^\alpha)$ of $\R[\x]$, let $L_\y:\R[\x]\to\R$ be the linear functional
\[f\quad (=\sum_{\alpha}f_{\alpha}\,\x^\alpha)\quad\mapsto\quad
L_\y(f)\,=\,\sum_{\alpha}f_{\alpha}\,y_{\alpha},\]
and let $\M_d(\y)$ be the symmetric matrix with rows and columns indexed in 
the canonical basis $(\x^\alpha)$, and defined by:
\begin{equation}
\label{moment}\M_d(\y)(\alpha,\beta)\,:=\,L_\y(\x^{\alpha+\beta})\,=\,y_{\alpha+\beta},\quad\alpha,\beta\in\N^n_d\end{equation}
with $\N^n_d:=\{\alpha\in\N^n\::\:\vert \alpha\vert \:(=\sum_i\alpha_i)\leq d\}$.

If $\y$ has a representing measure $\mu$, i.e., if
$y_\alpha=\int \x^\alpha d\mu$ for every $\alpha\in\N^n$, then
\[\langle \f,\M_d(\y)\f\rangle\,=\,\int f(\x)^2\,d\mu(\x)\,\geq\,0,\qquad \forall \,f\in\R[\x]_d,\]
and so $\M_d(\y)\succeq0$.
A measure $\mu$ is said to be moment determinate if there is no other measure with same moments. In particular, and as an easy consequence of the Stone-Weierstrass
theorem, every measure with compact support is determinate\footnote{To see this note 
that (a) two measures $\mu_1,\mu_2$
on a compact set $\K\subset\R^n$ are identical if and only if $\int_\K fd\mu_1=\int_\K fd\mu_2$ for all continuous 
functions $f$ on $\K$, and (b) by Stone-Weierstrass, the polynomials are dense 
in the space of continuous functions
for the sup-norm.}.

Not every sequence $\y$ satisfying $\M_d(\y)\succeq0$, $d\in\N$, has a representing measure. However:
\begin{prop}[Berg \cite{berg}]
\label{prop-berg}
Let $\y=(y_\alpha)$ be such that $\M_d(\y)\succeq0$, for every $d\in\N$. Then:

{\rm (a)} The sequence $\y$ has a representing measure
whose support is contained in the ball $[-a,a]^n$ if there exists $a,c>0$ such that $\vert y_\alpha\vert\leq c\, a^{\vert\alpha\vert}$ for every $\alpha\in\N^n$.

{\rm (b)} The sequence $\y$ has a representing measure $\mu$ on $\R^n$ if
\begin{equation}
\label{carleman}
\sum_{t=1}^\infty L_\y(x_i^{2t})^{-1/2t}\,=\,+\infty,\qquad \forall \,i=1,\ldots,n.
\end{equation}
In addition, in both cases (a) and (b) the measure $\mu$ is moment determinate.
\end{prop}
Condition (b) is an extension to the multivariate case of Carleman's condition in the univariate case and is due to Nussbaum \cite{nussbaum}. For more details see e.g. Berg
\cite{berg} and/or Maserick and Berg \cite{maserick}.

\subsection*{Localizing matrix} Similarly, with $\y=(y_{\alpha})$
and $f\in\R[\x]$ written
\[\x\mapsto f(\x)\,=\,\sum_{\gamma\in\N^n}f_{\gamma}\,\x^\gamma,\]
let $\M_d(f\,\y)$ be the symmetric matrix with rows and columns indexed in 
the canonical basis $(\x^\alpha)$, and defined by:
\begin{equation}
\label{localizing}
\M_d(f\,\y)(\alpha,\beta)\,:=\,L_\y\left(f(\x)\,\x^{\alpha+\beta}\right)\,=\,\sum_{\gamma}f_{\gamma}\,
y_{\alpha+\beta+\gamma},\qquad\forall\,
\alpha,\beta\in\N^n_d.\end{equation}
If $\y$ has a representing measure $\mu$,
then $\langle \g,\M_d(f\,\y)\g\rangle=\int g^2fd\mu$,
and so if $\mu$ is supported on the set $\{\x\,:\,f(\x)\geq0\}$,
then $\M_d(f\,\y)\succeq0$ for all $d=0,1,\ldots$ because
\begin{equation}
\label{localizing2}
\langle \g,\M_d(f\,\y)\g\rangle\,=\,\int g(\x)^2f(\x)\,d\mu(\x)\,\geq\,0,\qquad \forall \,g\in\R[\x]_d.\end{equation}

\section{Nonnegativity on closed sets}

Recall that if $\X$ is a separable metric space 
with Borel $\sigma$-field $\mathcal{B}$, the support ${\rm supp}\,\mu$ of a Borel measure $\mu$ on $\X$
is the (unique) smallest closed set $B\in\mathcal{B}$
such that $\mu(X\setminus B)=0$. Given a Borel measure $\mu$ on $\R^n$
and a measurable function $f:\R^n\to\R$, the mapping $B\mapsto \nu(B):=\int_Bfd\mu$, $B\in\mathcal{B}$, 
defines a set function on $\mathcal{B}$. 
If $f$ is nonnegative then $\nu$ is a Borel measure (which is finite if $f$ is $\mu$-integrable); 
see e.g. Royden \cite[p. 276 and p. 408]{royden}. 
If $f$ is not nonnegative then 
setting $B_1:=\{\x\,:\,f(\x)\geq0\}$ and $B_2:=\{\x\,:\,f(\x)<0\}$, 
the set function $\nu$ 
can be written as the difference 
\begin{equation}
\label{decomp}
\nu(B)\,=\,\nu_1(B)-\nu_2(B),\qquad B\in\mathcal{B},\end{equation}
of the two positive Borel measures $\nu_1,\nu_2$ defined by
\begin{equation}
\label{hahn}
\nu_1(B)=\int _{B_1\cap B}fd\mu,\quad
\nu_2(B)=-\int _{B_2\cap B}fd\mu,\quad\forall B\in\mathcal{B}.\end{equation}
Then $\nu$ is a {\it signed} Borel measure provided that either $\nu_1(B_1)$ or $\nu_2(B_2)$ is finite; 
see e.g. Royden \cite[p. 271]{royden}.
We first provide the following auxiliary result which is also 
of self-interest.

\begin{lem}
\label{aux}
Let $\X$ be a separable metric space, 
$\K\subseteq\X$ a closed set, and $\mu$ a 
Borel measure on $\X$ with $\supmu=\K$. 
A continuous function $f:\X\to\R$ is nonnegative on $\K$ if and only if the set function $B\mapsto \nu(B) =\int_{\K\cap B}fd\mu$, $B\in\mathcal{B}$, is a positive measure.
\end{lem}
\begin{proof}
The {\it only if part} is straightforward.  For the {\it if part},
if $\nu$ is a positive measure then 
$f(\x)\geq0$ for $\mu$-almost all $\x\in\K$.  That is, there is a Borel set $\G\subset\K$
such that $\mu(\G)=0$ and $f(\x)\geq0$ on $\K\setminus \G$. Indeed, otherwise suppose
that there exists a Borel set $B_0$ with $\mu(B_0)>0$ and $f<0$ on $B_0$;
then one would get the contradiction that $\nu$ is not positive because $\nu(B_0)=\int_{B_0}fd\mu<0$.
In fact, $f$ is called the Radon-Nikodym derivative of $\nu$
with respect to $\mu$; see Royden \cite[Theorem 23, p. 276]{royden}.

Next, observe that 
$\overline{\K\setminus\G}\subset\K$ and 
$\mu(\overline{\K\setminus\G})=\mu(\K)$.
Therefore $\overline{\K\setminus\G}=\K$ because
$\supmu\,(=\K)$ is the unique smallest closed set such that
$\mu(\X\setminus\K)=0$.
Hence, let $\x\in\K$ be fixed, arbitrary. 
As $\K=\overline{\K\setminus\G}$, there is a sequence 
$(\x_k)\subset \K\setminus\G$, $k\in\N$, with $\x_k\to\x$ as $k\to\infty$.
But since $f$ is continuous and $f(\x_k)\geq0$ for every $k\in\N$, 
we obtain the desired result $f(\x)\geq0$. 
\end{proof}
Lemma \ref{aux} itself (of which we have not been able to find a trace in the literature) is a characterization 
of nonnegativity on $\K$ for a continuous function $f$ on $\X$. However, one goal of this paper is to provide a more concrete characterization. To do so we first consider the case of a compact set $\K\subset\R^n$.

\subsection{The compact case}
Let $\K$ be a compact subset of $\R^n$. 
For simplicity, and with no loss of generality, we may and will assume that $\K\subseteq [-1,1]^n$. 

\begin{thm}
\label{thmradon}
Let $\K\subseteq [-1,1]^n$ be compact and 
let $\mu$ be an arbitrary, fixed, finite Borel measure on $\K$ with ${\rm supp}\,\mu=\K$, and with vector of moment
$\y=(y_\alpha)$, $\alpha\in\N^n$.  Let $f$ be a continuous function on $\R^n$. Then:

{\rm (a)} $f$ is nonnegative on $\K$ if and only if
\begin{equation}
\label{thmradon-1}
\int_\K g^2\,f\,d\mu\,\geq\,0,\qquad\forall \,g\in\R[\x],
\end{equation}
or, equivalently, if and only if
\begin{equation}
\label{thmradon-2}
\M_d(\z)\,\succeq\,0,\qquad d=0,1,\ldots
\end{equation}
where $\z=(z_\alpha)$, $\alpha\in\N^n$, with $z_\alpha=\int \x^\alpha f(\x)d\mu(\x)$, and with $\M_d(\z)$ as in (\ref{moment}). 

If in addition $f\in\R[\x]$ then 
(\ref{thmradon-2}) reads $\M_d(f\,\y)\succeq0$, $d=0,1,\ldots$,
where $\M_d(f\,\y)$ is the localizing matrix defined in 
(\ref{localizing}).\\

{\rm (b)} If in addition to be continuous, $f$ is also concave on $\K$, then 
$f$ is nonnegative on the convex hull $\co(\K)$ of $\K$ if and only if
(\ref{thmradon-1}) holds.
\end{thm}
\begin{proof}
The {\it only if} part is straightforward. Indeed, if $f\geq0$ on $\K$ then $\K\subseteq
\{\x\,:\,f(\x)\geq0\}$ and so for any finite Borel measure $\mu$ on $\K$, $\int_\K g^2fd\mu\geq0$ for every $g\in\R[\x]$.
Next, if $f$ is concave and $f\geq0$ on $\co(\K)$
then $f\geq0$ on $\K$ and so the ``only if" part of (b) also follows.

{\it If part.} The set function
$\nu(B)=\int_Bfd\mu$, $B\in\mathcal{B}$, 
can be written as the difference $\nu=\nu_2-\nu_2$ of the two positive finite Borel measures $\nu_1,\nu_2$ described
in (\ref{decomp})-(\ref{hahn}),
where $B_1:=\{\x\in\K:f(\x)\geq0\}$ and $B_2:=\{\x\in\K:f(\x)<0\}$. 
As $\K$ is compact and $f$ is continuous,
both $\nu_1,\nu_2$ are finite, and so $\nu$ is a finite signed Borel measure;
see Royden \cite[p. 271]{royden}.
In view of Lemma \ref{aux} it suffices to prove that in fact $\nu$
is a finite and positive Borel measure. So let  $\z=(z_\alpha)$, $\alpha\in\N^n$, be the sequence
defined by:
\begin{equation}
\label{defnu}
z_\alpha\,=\,\int_\K \x^\alpha d\nu(\x)\,:=\,\int_\K \x^\alpha f(\x)d\mu(\x),\qquad\forall\alpha\in\N^n.\end{equation}
Every $z_\alpha$, $\alpha\in\N^n$,  is finite because $\K$ is compact and $f$ is continuous.
So the condition 
\[\int_\K g(\x)^2f(\x)\,d\mu(\x)\geq0,\qquad\forall f\in\R[\x]_d,\]
reads $\langle\g,\M_d(\z)\g\rangle\geq0$ for all
$\g\in\R^{s(d)}$, that is, $\M_d(\z)\succeq0$, where $\M_d(\z)$ is the moment matrix defined in (\ref{moment}). And so (\ref{thmradon-1}) implies
$\M_d(\z)\succeq0$ for every $d\in\N$. 
Moreover, as $\K\subseteq [-1,1]^n$, 
\[\vert z_\alpha\vert\,\leq\,c\,:=\,\int_\K\vert f\vert d\mu,\qquad\forall\,\alpha\in\N^n.\]
Hence, by Proposition \ref{prop-berg}, $\z$ is the moment sequence of a finite (positive) Borel measure $\psi$ on $[-1,1]^n$, that is, as ${\rm supp}\,\nu\subseteq\K\subseteq [-1,1]^n$,
\begin{equation}
\label{equality}
\int_{[-1,1]^n} \x^\alpha\,d\nu(\x)\,=\,\int_{[-1,1]^n} \x^\alpha \,d\psi(\x),\qquad\forall\alpha\in\N^n.\end{equation}
But then using (\ref{decomp}) and (\ref{equality}) yields
\[\int_{[-1,1]^n} \x^\alpha\,d\nu_1(\x)\,=\,
\int_{[-1,1]^n} \x^\alpha \,d(\nu_2+\psi)(\x),\qquad\forall\alpha\in\N^n,\]
which in turn implies $\nu_1=\nu_2+\psi$ because
measures on a compact set are determinate.
Next, this implies $\psi=\nu_1-\nu_2\:(=\nu)$ and so
$\nu$ is a positive Borel 
measure on $\K$. Hence by Lemma \ref{aux},
$f(\x)\geq0$ on $\K$.

If in addition $f\in\R[\x]$, the sequence $\z=(z_\alpha)$ is obtained as a linear combination of $(y_\alpha)$. Indeed if $f(\x)=\sum_\beta f_\beta\,\x^\beta$ then
\[z_\alpha\,=\,\sum_{\beta\in\N^n}f_\beta\,y_{\alpha+\beta},
\qquad\forall\,\alpha\in\N^n,\]
and so in (\ref{thmradon-2}), $\M_d(\z)$ is nothing less than
the {\it localizing} matrix $\M_d(f\,\y)$ 
associated with $\y=(y_\alpha)$ and $f\in\R[\x]$,
defined in (\ref{localizing}), and (\ref{thmradon-2}) reads $\M_d(f\,\y)\succeq0$ for all $d=0,1,\ldots$

Finally, if $f$ is concave then $f\geq0$ on $\K$ 
implies $f\geq0$ on $\co(\K)$, and so the {\it only if} part of (b) also follows.
\end{proof}
Therefore, to check whether a polynomial $f\in\R[\x]$ is nonnegative on $\K$, it suffices to check if
every element of the countable family of real symmetric matrices $(\M_d(f\,\y))$, $d\in\N$,
is positive semidefinite.
\begin{rem}
{\rm An informal alternative proof of Theorem \ref{thmradon} which does not use Lemma \ref{aux} is as follows. If $f$ is not nonnegative
on $\K$ there exists  $\a\in\K$ such that $f(\a)<0$, and so as $\K$ is compact, there is a continuous function, e.g,
$\x\mapsto h(\x):=\exp(-c\Vert \x-\a\Vert ^2)$ close to $1$ in some open neighborhood $\B(\a,\delta)$ of $\a$,
and very small in the rest of $\K$. 
By the Stone-Weierstrass's theorem, one may choose $h$ to be a polynomial.
Next, the complement $\B(\a,\delta)^c\,(=\R^n\setminus \B(\a,\delta)$) of $\B(\a,\delta)$ is closed, and so
$\K\cap\B(\a,\delta)^c$ is a closed set contained in $\K$ (hence smaller than $\K$). Therefore
$\mu(\B(\a,\delta))>0$ because otherwise $\mu(\K\cap\B(\a,\delta)^c)=\mu(\K)$ which would imply
that $\K\cap\B(\a,\delta)^c$ is a support of $\mu$ smaller than $\K$, in contradiction with 
$\supmu=\K$. Hence, we would get the contradiction
\[\int h^2\,f\,d\mu\,\approx\,h(\a)^2f(\a)\,\mu(\B(\a,\delta))\,<\,0.\]
However, in the non compact case described in the next section, this argument is not valid.
}\end{rem}

\subsection{The non-compact case}

We now consider the more delicate case
where $\K$ is a closed set of $\R^n$, not necessarily compact. 
To handle arbitrary non compact sets $\K$ and arbitrary continuous functions $f$,
we need a reference measure $\mu$ with $\supmu=\K$ and with nice properties so that
integrals such as $\int g^2fd\mu$, $g\in\R[\x]$, are well-behaved.

So, let $\varphi$ be an arbitrary finite Borel measure on $\R^n$ whose support is exactly $\K$, and let $\mu$ be the finite Borel measure defined by:
\begin{equation}
\label{defmu}
\mu(B)\,:=\,\int_B\exp\left(-\sum_{i=1}^n\vert x_i\vert\right)\,d\varphi(\x),\qquad \forall\,B\in\mathcal{B}(\R^n).\end{equation}
Observe that $\supmu=\K$ and $\mu$ satisfies Carleman's condition (\ref{carleman}). Indeed,
let $\z=(z_\alpha)$, $\alpha\in\N^n$, be the sequence of moments of $\mu$. Then 
for every $i=1,\ldots,n$, and every $k=0,1,\ldots$, using $x_i^{2k}\leq (2k)\l\exp{\vert x_i\vert}$,
\begin{equation}
\label{bigM}
L_\z(x_i^{2k})=\int_\K x_i^{2k}\ d\mu(\x)\,\leq\, (2k)\l\,\int_\K\e^{\vert x_i\vert}\,d\mu(\x)\,\leq\,(2k)\l\,\varphi(\K)\,=:(2k)\l\,M.\end{equation}
Therefore for every $i=1,\ldots,n$, using $(2k)\l < (2k)^{2k}$ for every $k$, yields
\[\sum_{k=1}^\infty L_\z(x_i^{2k})^{-1/2k}\,\geq\,\sum_{k=1}^\infty M^{-1/2k}\,((2k)\l)^{-1/2k}
\geq\,\sum_{k=1}^\infty \frac{M^{-1/2k}}{2k}\,=\,+\infty,\]
i.e., (\ref{carleman}) holds.
Notice also that all the moments of $\mu$ (defined in (\ref{defmu})) are finite, and so every polynomial 
is $\mu$-integrable.

\begin{thm}
\label{noncompact}
Let $\K\subseteq \R^n$ be closed and let 
$\varphi$ be an arbitrary finite Borel measure whose support is exactly $\K$. 
Let $f$ be a continuous function on $\R^n$. 
If $f\in\R[\x]$ (i.e., $f$ is a polynomial) let $\mu$ be as in (\ref{defmu}) whereas if $f$ is not a polynomial
let $\mu$ be defined by
\begin{equation}
\label{noncompact-1}
\mu(B)\,:=\,\int_B \frac{\exp\left(-\sum_{i=1}^n\vert x_i\vert\right)}{1+f(\x)^2}\,d\varphi(\x),\qquad \forall\, B\in\mathcal{B}(\R^n).
\end{equation}
Then (a) and (b) of Theorem \ref{thmradon}  hold.
\end{thm}
For a detailed proof see \S \ref{appendix}. 

It is important to emphasize that in Theorem \ref{thmradon} and \ref{noncompact}, the set $\K$ is an {\it arbitrary} closed set of $\R^n$,
 and to the best of our knowledge, the characterization of nonnegativity of $f$ in terms of positive definiteness
 of the moment matrices $\M_d(\z)$ is new. But of course, this characterization becomes even more interesting
 when  one knows how the compute the moment sequence $\z=(z_\alpha)$, $\alpha\in\N^n$, which 
 is possible in a few special cases only. 
 
Important particular cases of nice such sets $\K$ include boxes, hyper rectangles, ellipsoids, 	and simplices
in the compact case, and the positive orthant,
or the whole space $\R^n$ in the non compact case. For instance,
for the whole space $\K=\R^n$ 
one may choose for $\mu$ in (\ref{defmu}) the multivariate Gaussian (or normal) probability measure
\[\mu(B)\,:=\,(2\pi)^{-n/2}\int_B \exp(-\frac{1}{2}\Vert \x\Vert^2)\,d\x,\qquad B\in\mathcal{B}(\R^n),\]
which the $n$-times product of the one-dimensional normal distribution
\[\mu_i(B)\,:=\,\frac{1}{\sqrt{2\pi}}\,\int_B \exp(-x_i^2/2)\,dx_i,\qquad B\in\mathcal{B}(\R),\]
whose  moments are all easily available in closed form. In Theorem \ref{noncompact} this corresponds to the choice
\begin{equation}
\label{normalchoice}
\varphi(B)\,=\,(2\pi)^{-n/2}\int_B \frac{\exp(-\Vert \x\Vert^2/2)}{\exp(-\sum_{i=1}^n \vert x_i\vert)}\,d\x,\qquad B\in \mathcal{B}(\R^n).\end{equation}
When $\K$ is the positive orthant $\R^n_+$ one may choose for $\mu$ the exponential probability measure
\begin{equation}
\label{exponential}
\mu(B)\,:=\,\int_B \exp(-\sum_{i=1}^n x_i)\,d\x,\qquad B\in\mathcal{B}(\R^n_+),\end{equation}
which the $n$-times product of the one-dimensional exponential distribution
\[\mu_i(B)\,:=\,\int_B \exp(-x_i)\,dx_i,\qquad B\in\mathcal{B}(\R_+),\]
whose  moments are also easily available in closed form. In Theorem \ref{noncompact} this corresponds to the choice
\[\varphi(B)\,=\,2^n\int_B \exp(-\sum_{i=1}^n x_i)\,d\x,\qquad B\in \mathcal{B}(\R^n_+).\]

\subsection{The cone of nonnegative polynomials}

The convex cone $\mathcal{C}_d\subset\R[\x]_{2d}$ of nonnegative polynomials of degree at most $2d$ (a nonnegative polynomial has necessarily even degree) is much harder to characterize
than its subcone $\Sigma[\x]_d$ of sums of squares. Indeed, while the latter has a simple semidefinite 
representation with lifting
(i.e. $\Sigma[\x]_d$ is the projection in $\R^{s(2d)}$
of a spectrahedron\footnote{A spectrahedron is the intersection of the cone of positive semidefinite matrices with an 
affine-linear space. Its algebraic representation is called a Linear Matrix Inequality (LMI).} in a higher dimensional space),
so far there is no such simple representation for the former.
In addition, when $d$ is fixed, 
Blekherman \cite{blek} has shown that after proper normalization, the ``gap" between $\mathcal{C}_d$ and $\Sigma[\x]_d$ increases unboundedly with the number of variables. 

We next provide a convergent hierarchy of
(outer) semidefinite approximations $(\mathcal{C}_d^k)$, $k\in\N$, of $\mathcal{C}_d$ where each $\cc^k_d$
has a semidefinite representation with {\it no} lifting (i.e., no projection is needed and $\cc^k_d$ is a spectrahedron). To the best of our knowledge, this is the first result of this kind.

Recall that with every $f\in\R[\x]_d$ is associated its vector of coefficients $\f=(f_\alpha)$, $\alpha\in\N^n_d$, in the canonical basis of monomials, and conversely, with every $\f\in\R^{s(d)}$ is associated a polynomial $f\in\R[\x]_d$ with vector of coefficients
$\f=(f_\alpha)$ in the canonical basis. Recall that for every $k=1,\ldots$,
\[\gamma_p:=\frac{1}{\sqrt{2\pi}}\,\int_{-\infty}^\infty x^p\,\e^{-x^2/2}\,dx\,=\,
\left\{\begin{array}{rl}
0&\mbox{if $p=2k+1$,}\\
\prod_{j=1}^k(2j-1)&\mbox{if $p=2k$,}\end{array}\right.\]
as $\gamma_{2k}=(2k-1)\gamma_{2(k-1)}$ for every $k\geq1$.
\begin{cor}
\label{th-cone}
Let $\mu$ be the probability measure on
$\R^n$ which is the $n$-times product of
the normal distribution on $\R$,
and so with moments $\y=(y_\alpha)$, $\alpha\in\N^n$,
\begin{equation}
\label{normal}
y_\alpha\,=\,\int_{\R^n}\x^\alpha\,d\mu\,=\,\prod_{i=1}^n
\left(\frac{1}{\sqrt{2\pi}}\int_{-\infty}^\infty  x^{\alpha_i}\e^{-x^2/2}dx\right),\qquad\forall\,\alpha\in\R^n.
\end{equation}
For every $k\in\N$, let $\cc_d^k:=\{\f\in\R^{s(2d)}\,:\,\M_k(f\,\y)\succeq0\}$, where $\M_k(f\,\y)$ is the localizing matrix in (\ref{localizing}) associated with $\y$ and $f\in\R[\x]_{2d}$. Each 
$\cc^k_d$ is a closed convex cone and a spectrahedron.

Then $\cc^0_d\supset\cc^1_d\cdots\supset\cc^k _d
\cdots\supset\cc_d$ and $f\in\cc_d$ if and only if 
its vector of coefficients $\f\in\R^{s(2d)}$ satisfies
$\f\in\cc^k _d$, for every $k=0,1,\ldots$.
\end{cor}
\begin{proof}
Following its definition (\ref{localizing}),
all entries of the localizing matrix $\M_k(f\,\y)$
are linear in $\f\in\R^{s(2d)}$, and so
$\M_k(f\,\y)\succeq0$ is an LMI. Therefore $\cc^k_d$
is a spectrahedron and a closed convex cone.
Next, let $\K:=\R^n$ and let $\mu$ be as in Corollary \ref{th-cone}
and so of the form (\ref{defmu}) with $\varphi$ as in (\ref{normalchoice}). 
Then $\mu$ satisfies Carleman's condition
(\ref{carleman}). Hence, by
Theorem \ref{noncompact} with $\K=\R^n$, $f$ is nonnegative on $\K$
if and only if (\ref{thmradon-2})  holds, which is equivalent to stating that $\M_k(f\,\y)\succeq0$, $k=0,1,\ldots$, which in turn is equivalent to stating that $\f\in\cc^k_d$, $k=0,1,\ldots$
\end{proof}
So the nested sequence of convex cones 
$\cc^0_d\supset\cc^k_d\cdots\supset \cc_d$ 
defines arbitrary close outer approximations of $\cc_d$.
In fact $\cap_{k=0}^\infty\cc^k_d$ is closed and
$\cc_d=\cap_{k=0}^\infty\cc^k_d$.
It is worth emphasizing that each $\cc^k_d$ is a 
spectrahedron with {\it no} lifting, that is,
$\cc^k_d$ is defined solely in terms of the vector of coefficients $\f$ with no additional variable (i.e.,
no projection is needed).

For instance, the first approximation $\cc^0_d$ is just the set
$\{\f\in\R^{s(2d)}:\,\int fd\mu\geq0\}$, which is a half-space of 
$\R^{s(2d)}$. And with $n=2$, 
\[\cc^1_d\,=\,\left\{\f\in\R^{s(d)}\::\:\left[\begin{array}{ccc}
\displaystyle\int fd\mu & \displaystyle\int x_1fd\mu & 
\displaystyle\int x_2fd\mu\\
\displaystyle\int x_1fd\mu & \displaystyle\int x_1^2fd\mu & 
\displaystyle\int x_1x_2fd\mu\\
\displaystyle\int x_2fd\mu & \displaystyle\int x_1x_2fd\mu & 
\displaystyle\int x_2^2fd\mu\end{array}\right]\,\succeq\,0\:\right\},\]
or, equivalently, $\cc^1_d$ is the convex basic semi-algebraic set:
\begin{eqnarray*}
\left\{\f\in\R^{s(2d)}\right.&:&\int fd\mu\,\geq\,0\\
&&\left(\int x_i^2fd\mu\right)\left(\int fd\mu\right)\,\geq\,\left(\int x_ifd\mu\right)^2,\quad i=1,2\\
&&\left(\int x_1^2fd\mu\right)\left(\int x_2^2fd\mu\right)\,\geq\,\left(\int x_1x_2fd\mu\right)^2\end{eqnarray*}
\[\left(\int fd\mu\right)\left[\left(\int x_1^2fd\mu\right)\left(\int x_2^2fd\mu\right)-\left(\int x_1x_2fd\mu\right)^2\right]-\]
\[\left(\int x_1fd\mu\right)^2\left(\int x_2^2fd\mu\right)
-\left(\int x_2fd\mu\right)^2\left(\int x_1^2fd\mu\right)+\]
\[\left.2\left(\int x_1fd\mu\right)\left(\int x_2fd\mu\right)
\left(\int x_1x_2fd\mu\right)\,\geq\,0\:\right\},\]
where we have just expressed the nonnegativity of
all principal minors of $\M_1(f\,\y)$.

A very similar result holds for the convex cone $\cc_d(\K)$
of polynomials of degree at most $d$, nonnegative on 
a closed set $\K\subset\R^n$.
  
\begin{cor}
\label{generalk}
Let $\K\subset\R^n$ be a closed set and let $\mu$ be defined
in (\ref{defmu}) where $\varphi$ is a an arbitrary finite Borel measure whose support is exactly $\K$.

For every $k\in\N$, let $\cc_d^k(\K):=\{\f\in\R^{s(d)}\,:\,\M_k(f\,\y)\succeq0\}$, where $\M_k(f\,\y)$ is the localizing matrix in (\ref{localizing}) associated with $\y$ and $f\in\R[\x]_{d}$. Each $\cc^k_d(\K)$ is a closed convex cone and a spectrahedron.

Then $\cc^0_d(\K)\supset\cc^1_d(\K)\cdots\supset\cc^k _d(\K)
\cdots\supset\cc_d(\K)$ and $f\in\cc_d(\K)$ if and only if 
its vector of coefficients $\f\in\R^{s(d)}$ satisfies
$\f\in\cc^k _d(\K)$, for every $k=0,1,\ldots$.
\end{cor}
The proof which mimicks that of Corollary \ref{th-cone} is omitted. Of course, for practical computation, one is
restricted to sets $\K$ where one may
compute {\it effectively} the moments of the measure $\mu$.
An example of such a set $\K$ is the positive orthant, in which case one may choose 
the measure $\mu$ in (\ref{exponential})
for which all moments are explicitly available. For compact sets $\K$
let us mention balls, boxes, ellipsoids, and simplices. But again, any compact set
where one knows how to compute all moments of some measure with support exactly $\K$, is fine.

To the best of our knowledge this is the first characterization of an outer approximation of 
the cone $\cc_d(\K)$ in a relatively general context. Indeed,
for the basic semi-algebraic set
\begin{equation}
\label{basic}
\K\,=\,\{\x\in\R^n\::\: g_j(\x)\,\geq\,0,\quad j=1,\ldots,m\:\},
\end{equation}
Stengle's Nichtnegativstellensatz \cite{stengle} states that $f\in\R[\x]$ is nonnegative on $\K$ if and only if
\begin{equation}
\label{stengle1}
p\,f\,=\,f^{2s}+q,\end{equation}
for some integer $s$ and polynomials $p,q\in P(g)$, where $P(g)$ is the {\it preordering}\footnote{The preordering $P(g)$
associated with the $g_j$'s is the set of polynomials of the form $\sum_{\alpha\in\{0,1\}^m}\sigma_\alpha\, g_1^{\alpha_1}\cdots g_m^{\alpha_m}$
where $\sigma_\alpha\in\Sigma[\x]$ for each $\alpha$.} associated with the $g_j$'s.
In addition, there exist bounds on the integer $s$ and the degree of the s.o.s. weights in the 
definition of $p,q\in P(g)$, so that in principle, when $f$ is {\it known}, checking whether $f\geq0$ on $\K$ reduces to solving a single SDP
to compute $h,p$ in the nonnegativity certificate (\ref{stengle1}). However, the size of this SDP is potentially huge
and makes it unpractical. Moreover, the representation of $\cc_d(\K)$ in (\ref{stengle1}) is not 
convex in the vector of coefficients of $f$ because it involves
$f^{2s}$ as well as the product $p f$.

\begin{rem}
{\rm If in Corollary \ref{generalk} one replaces the finite-dimensional
convex cone $\cc_d(\K)\subset\R[\x]_d$ with the infinite-dimensional convex cone $\cc(\K)\subset\R[\x]$ of {\it all}
polynomials nonnegative on $\K$, and $\cc^k_d(\K)\subset\R[\x]_d$ with
$\cc^k(\K)=\{\f\in\R^{s(2k)}:\,\M_k(f\,\y)\succeq0\}$, then the nested sequence of (increasing but finite-dimensional)
convex cones $\cc^k(\K)$, $k\in\N$,
provides finite-dimensional approximations of $\cc(\K)$.
}\end{rem}

\section{Application to polynomial optimization}

Consider the polynomial optimization problem
\begin{equation}
\label{poly-opt}
\P:\quad f^*\,=\,\displaystyle\inf_\x\{\: f(\x)\::\: \x\in\K\:\},
\end{equation}
where $\K\subseteq\R^n$ is closed and 
$f\in\R[\x]$. 

If $\K$ is compact let $\mu$ be a finite Borel measure with $\supmu=\K$
and if $\K$ is not compact,
let $\varphi$ be an arbitrary finite Borel measure with $\supphi=\K$
and let $\mu$ be as in (\ref{defmu}).
In both cases, the sequence of moments $\y=(y_\alpha)$, $\alpha\in\N^n$, is well-defined,
and we assume that $y_\alpha$ is available or can be computed, for every $\alpha\in\N^n$.

Consider the sequence of semidefinite programs:
\begin{equation}
\label{primal}
\lambda_d=\sup_{\lambda\in\R}\,\{\lambda\::\: \M_d(f_\lambda\,\y)\succeq0\:\}
\end{equation}
where $f_\lambda\in\R[\x]$ is the polynomial $\x\mapsto f(\x)-\lambda$. Notice that (\ref{primal}) has only one variable!
\begin{thm}
\label{thm-optim}
Consider the hierarchy of semidefinite programs (\ref{primal})
indexed by $d\in\N$. Then:

{\rm (i)} (\ref{primal}) has an optimal solution 
$\lambda_d\geq f^*$ for every $d\in\N^n$.

{\rm (ii)} The sequence $(\lambda_d)$, $d\in\N$, is monotone nonincreasing and $\lambda_d\downarrow f^*$ as $d\to\infty$.
\end{thm}
\begin{proof}
(i) Since $f-f^*\geq0$ on $\K$,
by Theorem \ref{thmradon}, 
$\lambda:=f^*$ is a feasible solution of (\ref{primal}) for every $d$. Hence $\lambda_d\geq f^*$ for every $d\in\N$.
Next, let $d\in\N$ be fixed, and let $\lambda$ be an arbitrary feasible solution of (\ref{primal}).
From the condition $\M_d(f_\lambda\,\y)\succeq0$, the diagonal entry $\M_d(f_\lambda\,\y)(1,1)$ must be nonnegative, i.e.,
$\lambda y_0\leq \sum_\alpha f_\alpha\,y_\alpha$, and so, as we maximize and $y_0>0$, (\ref{primal}) must have an optimal solution 
$\lambda_d$.

(ii) Obviously $\lambda_d\leq\lambda_m$ whenever 
$d\geq m$, because $\M_d(f_\lambda\,\y)\succeq0$
implies $\M_m(f_\lambda\,\y)\succeq0$. Therefore, the sequence $(\lambda_d)$, $d\in\N$, is monotone
nonincreasing and being bounded below by $f^*$, converges to $\lambda^*\geq f^*$. Next, suppose that $\lambda^*>f^*$; fix $k\in\N$, arbitrary. The convergence 
$\lambda_d\downarrow \lambda^*$ implies $\M_k(f_{\lambda^*}\,\y)\succeq0$. As $k$ was arbitrary, we obtain that
$\M_d(f_{\lambda^*}\,\y)\succeq0$ for every $d\in\N$. 
But then by Theorem \ref{thmradon} or Theorem \ref{noncompact}, $f-\lambda^*\geq0$ on
$\K$, and so $\lambda^*\leq f^*$, in contradiction with $\lambda^*>f^*$. Therefore $\lambda^*=f^*$.
\end{proof}
For each $d\in\N$, the semidefinite program (\ref{primal}) provides
an upper bound on the optimal value $f^*$ only. 
We next show that the dual contains some
information on global minimizers, at least when
$d$ is sufficiently large.

\subsection{Duality}
\label{sec-duality}
Let $\mathcal{S}_d$ be the space of real symmetric $s(d)\times s(d)$ matrices.
One may write the semidefinite program (\ref{primal}) as
\begin{equation}
\label{eigenvalue}
\lambda_d=\sup_\lambda\{\lambda\,:\,
\lambda\M_d(\y)\preceq \,\M_d(f\,\y)\},\end{equation}
which in fact is a generalized eigenvalue problem for the pair of matrices
$\M_d(\y)$ and $\M_d(f\,\y)$. Its dual is the semidefinite program
\[\inf_{\X\in\mathcal{S}_d}\{\:\langle \X,\M_d(f\,\y)\rangle\,:\,\langle \X,\M_d(\y)\rangle\,=\,1;\:\X\succeq0\:\},\]
or, equivalently,
\begin{equation}
\label{dual}
\lambda^*_d=\inf_\sigma\:\left\{\int_\K f\,\sigma \,d\mu\::\:\int_\K \sigma\,d\mu=1;\:\sigma\in\Sigma[\x]_d\:\right\}.\end{equation}
So the dual problem (\ref{dual}) is to find a sum of squares 
polynomial $\sigma$ of degree at most $2d$ (normalized 
to satisfy $\int \sigma d\mu=1$) that minimizes the integral $\int f\sigma d\mu$, and a simple interpretation of (\ref{dual}) is as follows:

With $M(\K)$ being the space of Borel probability measures on 
$\K$, we know that $f^*=\displaystyle\inf_{\varphi\in M(\K)}\,\displaystyle\int_\K  fd\varphi$. Next, let $M_d(\mu)\subset M(\K)$ be the space of probability measures on $\K$ which have a density 
$\sigma\in\Sigma[\x]_d$ with respect to $\mu$. 
Then (\ref{dual}) reads 
$\displaystyle\inf_{\varphi\in M_d(\mu)}\,\displaystyle\int_\K  fd\varphi$, which clearly shows why one obtains an upper bound on 
$f^*$. Indeed, instead of searching in $M(\K)$
one searches in its subset $M_d(\mu)$.
What is not obvious at all is whether the obtained
upper bound obtained in (\ref{dual}) converges to $f^*$ when the degree
of $\sigma\in\Sigma[\x]_d$ is allowed to increase!
\begin{thm}
\label{thdual}
Suppose that $f^*>-\infty$ and 
$\K$ has nonempty interior. Then :

{\rm (a)} There is no duality gap between (\ref{primal}) and (\ref{dual})
and (\ref{dual}) has an optimal solution $\sigma^*\in\Sigma[\x]_d$ which satisfies
$\displaystyle\int_\K (f(\x)-\lambda_d)\,\sigma^*(\x)d\mu(\x)=0$.

{\rm (b)} If $\K$ is convex and $f$ is convex, let
$\x^*_d:=\int \x\,\sigma^*(\x)d\mu(\x)$. Then $\x^*_d\in\K$
and $f^*\leq f(\x^*_d)\leq\lambda_d$, so that $f(\x^*_d)\to f^*$ as $d\to\infty$. Moreover, if the set
$\{\x\in\K\,:\,f(\x)\leq f_0\}$ is compact for some $f_0>f^*$,
then any accumulation point $\x^*\in\K$ of the sequence $(\x^*_d)$, $d\in\N$,
is a minimizer of problem (\ref{poly-opt}), that is, $f(\x^*)=f^*$.
\end{thm}
\begin{proof}
(a) Any scalar $\lambda < f^*$ is a feasible solution
of (\ref{primal}) and in addition,
$\M_d((f-\lambda)\,\y)\succ0$ because since
$\K$ has nonempty interior and $f-\lambda>0$ on $\K$,
\[\langle \g,\M_d((f-\lambda)\,\y)\,\g\rangle\,=\,\int_\K 
(f(\x)-\lambda)g(\x)^2\mu(d\x)>0,\qquad\forall g\in\R[\x]_d.\]
But this means that Slater's condition\footnote{For an optimization problem $\inf_\x \{ f_0(\x)\,:\, f_j(\x)\geq0,\,j=1,\ldots,m\}$, Slater's condition states that there exists
$\x_0$ such that $f_j(\x_0)>0$ for every $j=1,\ldots,m$.}
 holds for (\ref{primal})
which in turn implies that there is no duality gap
and (\ref{dual}) has an optimal solution $\sigma^*\in\Sigma[\x]_d$; see e.g. \cite{boyd}. And so,
\[\int_\K (f(\x)-\lambda_d)\,\sigma^*(\x)\,d\mu(\x)=
\int_\K f\,\sigma^*\,d\mu-\lambda_d\,=\,0.\]
(b) Let $\nu$ be the Borel probability measure on $\K$ defined by $\nu(B)=\int_B\sigma^* d\mu$, $B\in\mathcal{B}$.
As $f$ is convex, by Jensen's inequality (see e.g. McShane \cite{jensen}),
\[\int_\K f\sigma^*d\mu\,=\,\int_\K fd\nu\geq\,f\left(\int_\K\x\,d\nu\right)\,=\,f(\x^*_d).\]
In addition, if $\K$ is convex then $\x^*_d\in\K$ and so,
$\lambda_d\geq f(\x^*_d)\geq f^*$. Finally if for some $f_0>f^*$,
the set $\H:=\{\x\in\K: f(\x)\leq f_0\}$ is compact, 
and since $\lambda_d\to f^*$, then $f(\x^*_d)\leq f_0$ for
$d$ sufficiently large, i.e., $\x^*_d\in\H$ for sufficiently large 
$d$. By compactness
there is a subsequence $(d_\ell)$, $\ell\in\N$, and a point $\x^*\in\K$ such that $\x^*_{d_\ell}\to\x^*$ as $\ell\to\infty$. 
Continuity of $f$ combined with
the convergence $f(\x^*_d)\to f^*$
yields $f(\x^*_{d_\ell})\to f(\x^*)= f^*$ as $\ell\to\infty$.
As the convergent subsequence $(\x^*_{d_\ell})$ was arbitrary, the proof is complete.
\end{proof}
So in case where $f$ is a convex polynomial and
$\K$ is a convex set, Theorem \ref{thdual} provides a means of approximating not only the optimal value $f^*$, but also a global minimizer $\x^*\in\K$. 

In the more subtle nonconvex case, one still can obtain some information on global minimizers from an optimal solution $\sigma^*\in\Sigma[\x]_d$ of (\ref{dual}).
Let $\epsilon>0$ be fixed,
and suppose that $d$ is large enough so that 
$f^*\leq\lambda_d\leq f^*+\epsilon$. Then, 
by Theorem \ref{dual}(a),
\[\int_\K (f(\x)-f^*)\sigma^*(\x)\,d\mu(\x)\,=\,\lambda_d-f^*\,<\,\epsilon.\]
As $f-f^*\geq0$ on $\K$, necessarily the measure $d\nu=\sigma^*d\mu$ gives very small weight to regions of
$\K$ where $f(\x)$ is significantly larger than $f^*$.
For instance, if $\epsilon=10^{-2}$ and
$\Delta:=\{\x\in\K: f(\x)\geq f^*+1\}$, then 
$\nu(\Delta)\leq 10^{-2}$, i.e., the set 
$\Delta$ contributes to less than $1\%$ of the total mass of 
$\nu$. So if $\mu$ is uniformly distributed on 
$\K$ (which is a reasonable choice if one has to compute all moments of $\mu$) then a simple inspection of the values of $\sigma^*(\x)$ provides some rough indication on 
where (in $\K$) $f(\x)$ is close to $f^*$.

The interpretation (\ref{dual}) of the dual shows that in general the monotone convergence is only asymptotic and cannot be finite.
Indeed if $\K$ has a nonempty interior then the probability measure $d\nu=\sigma d\mu$ cannot be a Dirac measure at 
any global minimizer $\x^*\in\K$. An exception is the discrete case, i.e., when $\K$
is a finite number of points, like in e.g. 0/1 programs. Indeed we get:
\begin{cor}
\label{exception}
Let $\K\subset \R^n$ be a discrete set $(\x(k))\subset\R^n$, $k\in J$, and let $\mu$ be the probability measure uniformly
distributed in $\K$, i.e.,
\[\mu\,=\,\frac{1}{s}\sum_{k=1}^s \delta_{\x(k)},\]
where $s=\vert J\vert$ and $\delta_\x$ denote the Dirac measure at the point $\x$. Then the optimal value 
$\lambda_d$ of (\ref{primal}) satisfies $\lambda_d=f^*$ for some integer $d$.
\end{cor}
\begin{proof}
Let $\x^*=\x(j^*)$ (for some $j^*\in J$)  be the global minimizer of $\min \{f(\x)\,:\,\x\in \K\}$.
For each $k=1,\ldots,s$ there exists a polynomial $q_k\in\R[\x]$ such that
$q_k(\x(j))=\delta_{k=j}$ for every $j=1,\ldots,s$ (where $\delta_{k=j}$ denotes the Kronecker symbol). The polynomials $(q_k)$ are called {\it interpolation} polynomials.
So let $\sigma^*:=sq_{j^*}^2\in\Sigma[\x]$, so that
\[\int_\K f(\x)\sigma^*(\x)d\mu(\x)\,=\,f(\x(j^*))\,=\,f^*\quad\mbox{and}\quad\int_\K \sigma^*d\mu\,=\,1.\]
Hence as soon as $d\geq {\rm deg}\,q_{j^*}$, $\sigma^*\in\Sigma[\x]_d$  is a feasible solution of (\ref{dual}), and so
from $f^*=\int f\sigma^*d\mu\geq\lambda^*_d \geq\lambda_d\geq f^*$ we deduce that $\lambda^*_d=\lambda_d=f^*$, the desired result.
\end{proof}

There are several interesting cases where the above described methodology can apply, i.e., cases where $\y$ can be obtained either explicitly
in closed form or numerically. In particular, when $\K$ is either:
\begin{itemize}
\item A box $\B:=\prod_{i=1}^n[a_i,b_i]\subset\R^n$, with
$\mu$ being the normalized Lebesgue measure on $\B$.
The sequence $\y=(y_\alpha)$ is trivial to obtain in closed form.
\item The discrete set $\{-1,1\}^n$ with $\mu$ being uniformly
distributed and normalized. Again the sequence $\y=(y_\alpha)$ is trivial to obtain in closed form. Notice that in this case we obtain a new 
hierarchy of semidefinite relaxations (with only one variable) for the celebrated MAXCUT problem (and any nonlinear 0/1 program).
\item The unit Euclidean ball $\B:=\{\x\,:\,\Vert \x\Vert^2\leq 1\}$ with $\mu$ uniformly distributed,
and similarly the unit sphere $\bS:=\{\x\,:\,\Vert \x\Vert^2= 1\}$, with $\mu$ being the rotation invariant probability measure on 
$\bS$. In both cases the moments $\y=(y_\alpha)$ are obtained easily. 
\item A simplex $\Delta\subset\R^n$,
in which case if one takes $\mu$ as the Lebesgue measure
then all moments of $\mu$ can be computed  numerically.
In particular, 
with $d$ fixed, this computation can be done in time polynomial time. See e.g. the recent work of \cite{deloera}. 
\item The whole space $\R^n$ in which case 
$\mu$ may be chosen
to be the product measure $\otimes_{i=1}^n \nu_i$
with each $\nu_i$ being the normal distribution.
Observe that one then obtains a 
new hierarchy of semidefinite approximations (upper bounds) for unconstrained global optimization. The corresponding monotone sequence of
upper bounds converges to $f^*$ no matter if the problem has a global minimizer or not. This may be an alternative and/or a complement to the recent convex relaxations provided in Schweighofer 
\cite{markus} and H\`a and Vui \cite{ha} which also work when $f^*$ is not attained, and provide a convergent sequence of {\it lower} bounds.
\item The positive orthant $\R^n_+$, in which case 
$\mu$ may be chosen to be the product measure $\otimes_{i=1}^n \nu_i$
with each $\nu_i$ being the exponential 
distribution $\nu_i(B)=\int_{\R^+\cap B}\e^{-x}dx$, $B\in\mathcal{B}$. In particular
if $\x\mapsto f(\x) :=\x^T\A\x$ where $\A\in\mathcal{S}_n$, then one obtains
a hierarchy of numerical tests to check whether $\A$ is a {\it copositive} matrix. Indeed,
if $\lambda_d$ is an optimal solution of (\ref{eigenvalue}) then
$\A$ is copositive if and only if $\lambda_d\geq0$ for all $d\in\N$.
Notice that we also obtain a hierarchy of {\it outer approximations} $({\rm COP}_d)\subset\mathcal{S}_n$
of the cone ${\rm COP}$ of $n\times n$ 
copositive matrices. Indeed, for every $\A\in\mathcal{S}_n$, let $f_\A$ be the quadratic form
$\x\mapsto f_\A(\x):=\x^T\A\x$. Then, for every $d$, the set
\[{\rm COP}_d\,:=\,\{\A\in\mathcal{S}_n\::\: \M_d(f_\A\,\y)\,\succeq0\:\}\]
is a convex cone defined only in terms of the coefficients of the matrix $\A$.
It is even a spectrahedron since $\M_d(f_\A\,\y)$ is a linear matrix inequality in the coefficients of $\A$.
And in view of Theorem \ref{thmradon}(a), ${\rm COP}\,=\,\bigcap_{d\in\N}{\rm COP}_d$.
\end{itemize}
\subsection{Examples}
In this section we provide three simple examples to illustrate the above methodology.
\begin{ex}
\label{ex1}
{\rm Consider the global minimization on $\K=\R^2_+$ of
the Motzkin-like polynomial $\x\mapsto f(\x)=x_1^2x_2^2\,(x_1^2+x_2^2-1)$ whose global minimum is $f^*=-1/27\approx-0.037$,
attained at $(x_1^*,x_2^*)=(\pm\sqrt{1/3},\pm\sqrt{1/3})$.
Choose for $\mu$ the probability measure 
$\mu(B):=\int_B\e^{-x_1-x_2}d\x$, $B\in\mathcal{B}(\R^2_+)$, for which the sequence of moments $\y=(y_{ij})$, $i,j\in\N$,
is easy to obtain. Namely $y_{ij}=i{\rm !}j{\rm !}$ for every
$i,j\in\N$.
Then the semidefinite relaxations (\ref{primal}) yield $\lambda_0=92$, 
$\lambda_1=1.5097$, and $\lambda_{14}=-0.0113$, showing a significant and rapid decrease in first iterations with a long tail close to $f^*$, illustrated in Figure \ref{fig1}.
Then after $d=14$, one encounters some numerical problems and we cannot trust the 
results anymore.

If we now minimize the same polynomial $f$ on the box $[0,1]^2$, one choose for $\mu$ the probability uniformly distributed on $[0,1]^2$,
whose moments $\y=(y_{ij})$, $i,j\in\N$, are also easily obtained by $y_{ij}=(i+1)^{-1}(j+1)^{-1}$.
Then one obtains $\lambda_0=0.222$, $\lambda_1=-0.055$, and $\lambda_{10}=-0.0311$, showing again a rapid decrease in first iterations with a long tail close to $f^*$, illustrated in Figure \ref{fig2}.
}
\begin{figure}[h!]
\begin{center}
\includegraphics[width=0.8\textwidth]{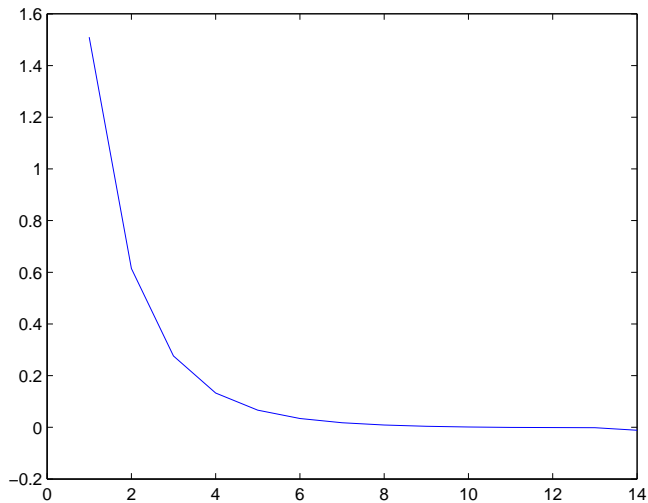}
\caption{Minimizing the Motzkin-like polynomial in $\R^2_+$\label{fig1}}
\end{center}
\end{figure}
\begin{figure}[h!]
\begin{center}
\includegraphics[width=0.8\textwidth]{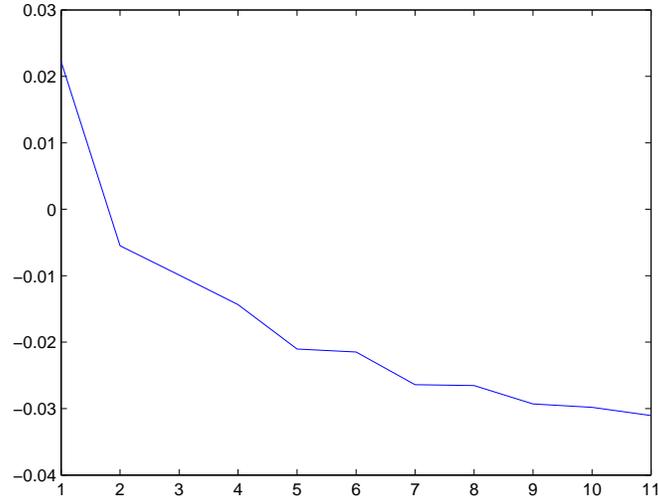}
\caption{Minimizing the Motzkin-like polynomial in $[0,1]^2$\label{fig2}}
\end{center}
\end{figure}

\end{ex}

\begin{ex}
\label{ex2}
{\rm Still on $\K=\R^2_+$, consider the global 
minimization of
the polynomial $\x\mapsto x_1^2+(1-x_1x_2)^2$
whose global minimum  $f^*=0$ is not attained. 
Again, choose for $\mu$ the probability measure 
$\mu(B):=\int_B\e^{-x_1-x_2}d\x$, $B\in\mathcal{B}(\R^2_+)$.
Then the semidefinite relaxations (\ref{primal}) yield 
$\lambda_0=5$, 
$\lambda_1=1.9187$ and 
$\lambda_{15}=0.4795$,
showing again a significant and rapid decrease in first iterations with a long tail close to $f^*$, illustrated in Figure \ref{fig3};
numerical problems occur after $d=15$.
However, this kind of problems where the global minimum $f^*$ is not attained,  is notoriously difficult.
Even the semidefinite relaxations defined in \cite{ha} (which provide lower bounds on $f^*$)
and especially devised for such problems, encounter numerical difficulties; see \cite[Example 4.8]{ha}.
\begin{figure}[h!]
\begin{center}
\includegraphics[width=0.8\textwidth]{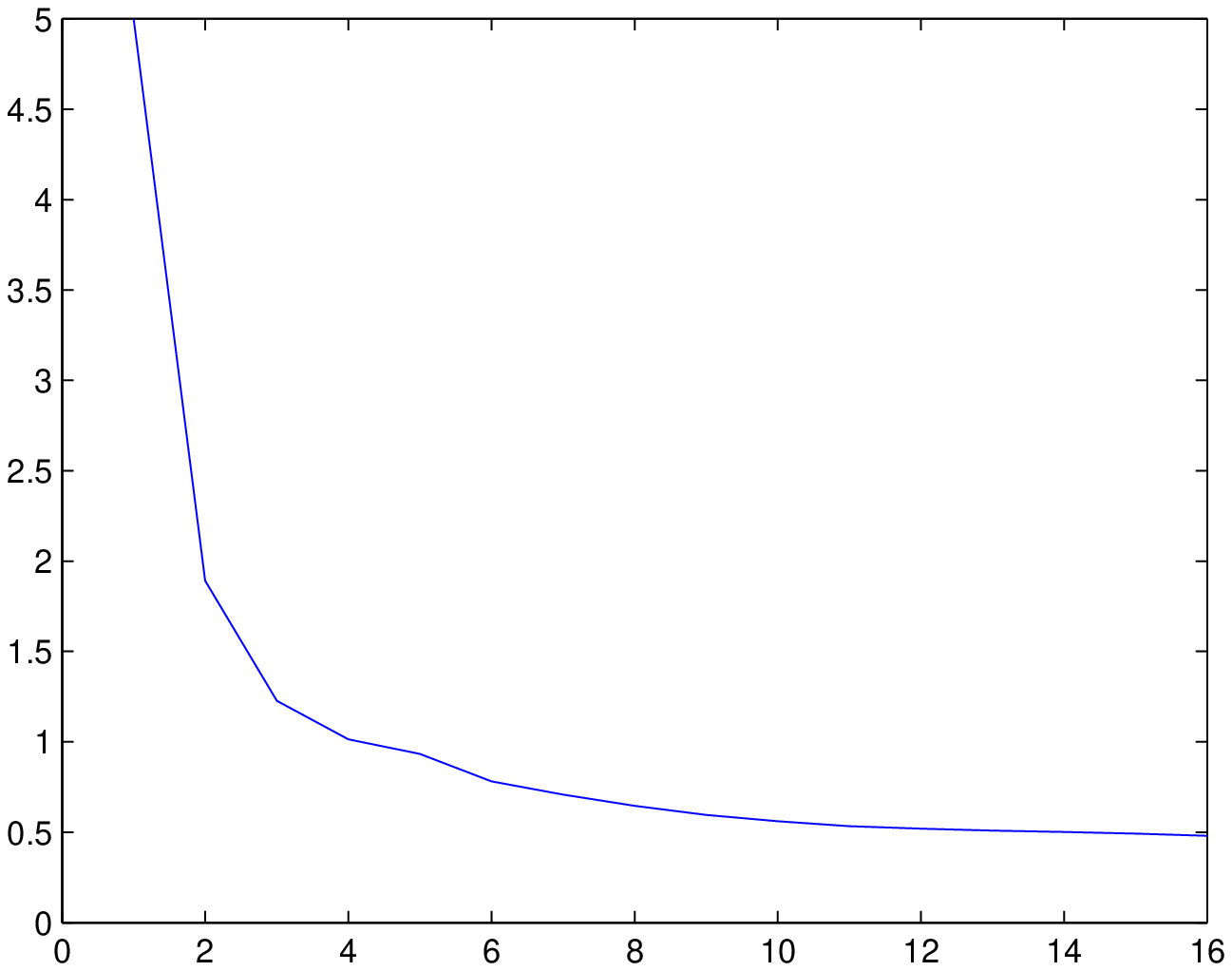}
\caption{Minimizing $x_1^2+(1-x_1x_2)^2$ on $\R^2_+$}
\label{fig3}
\end{center}
\end{figure}
}
\end{ex}
\begin{ex}
\label{ex4}
{\rm The following example illustrates the duality results of Section \S \ref{sec-duality}.
The univariate polynomial $x\mapsto f(x):=0.375-5x+21x^2-32x^3+16x^4$ displayed in Fig \ref{fig4}
has two global minima at $x^*_1=0.1939$ and $x^*_2=0.8062$, with $f^*=-0.0156$.
In Fig \ref{fig5} is plotted the sequence of upper bounds $\lambda_d\to f^*$ as $\d\to\infty$,
with again a rapid decrease in first iterations.
One has plotted in Fig \ref{fig6} the s.o.s. polynomial $x\mapsto \sigma(x)$,
optimal solution of (\ref{dual}) with $d=10$, associated with the  probability density $\sigma(x) dx$ as explained in
\S \ref{sec-duality}. As expected, two peaks appear at the points $\tilde{x}_i\approx x^*_i$, $i=1,2$.
\begin{figure}[h!]
\begin{center}
\includegraphics[width=0.8\textwidth]{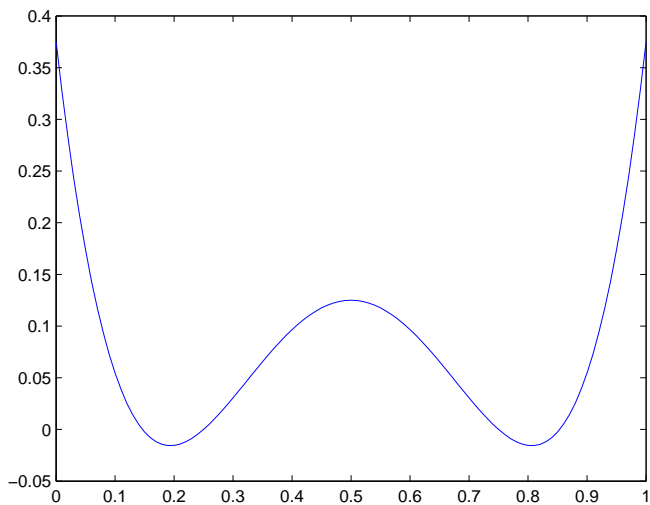}
\caption{$f(x)=0.375-5x+21x^2-32x^3+16x^4$ on $[0,1]$}
\label{fig4}
\end{center}
\end{figure}

\begin{figure}[h!]
\begin{center}
\includegraphics[width=0.8\textwidth]{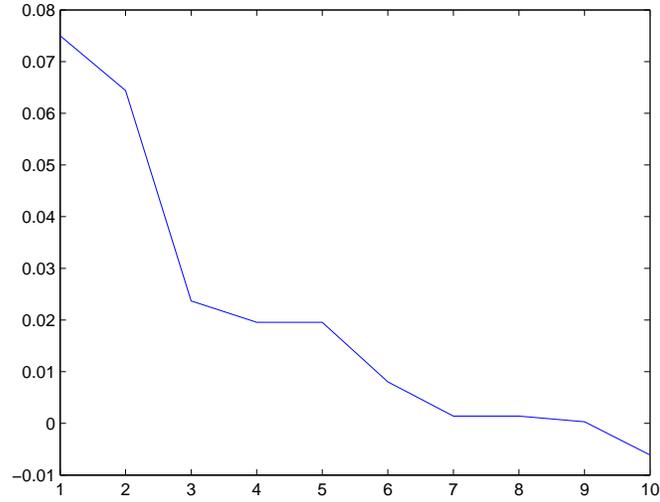}
\caption{Minimizing $0.375-5x+21x^2-32x^3+16x^4$ on $[0,1]$}
\label{fig5}
\end{center}
\end{figure}
\begin{figure}[h!]
\begin{center}
\includegraphics[width=0.8\textwidth]{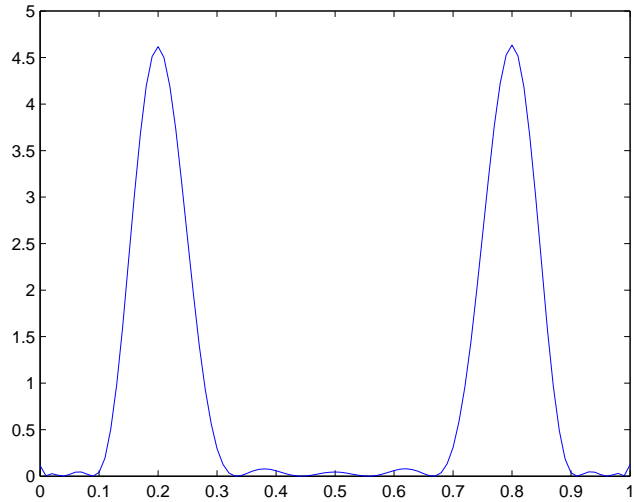}
\caption{The probability density $\sigma(x)dx$ on $[0,1]$}
\label{fig6}
\end{center}
\end{figure}
}\end{ex}
\begin{ex}
\label{maxcut}
{\rm We finally consider a discrete optimization problem, namely the celebrated MAXCUT problem 
\[f^*\,=\,\min_\x \{\x^T\Q\x\::\:\x\in\{-1,1\}^n\},\]
where $\Q=(Q_{ij})\in\R^{n\times n}$ is a real symmetric matrix whose all diagonal elements vanish.
The measure $\mu$ is uniformly distributed on $\{-1,1\}^n$ so that its moments are readily available.
We first consider the equal weights case, i.e., $Q_{ij}=1/2$ for all $(i,j)$ with $i\neq j$ in which case $f^*=-\lfloor n/2\rfloor$.
With $n=11$  the successive values for $\lambda_d$, $d\leq4$, are displayed in Table \ref{table1} and 
$\lambda_4$ is relatively close to $f^*$. Next we have generated five random instances of MAXCUT
with $n=11$ but $Q_{ij}=0$ with probability 1/2, and if $Q_{ij}\neq0$ it is randomly generated 
using the Matlab ``rand" function.
The successive values of $\lambda_d$, $d\leq 4$, are displayed in Table \ref{table2}, and again,
$\lambda_4$ is quite close to $f^*$\footnote{The optimal value $f^*$ has been computed using the GloptiPoly software
\cite{gloptipoly} dedicated to solving the Generalized Problem of Moments}.

\begin{table}
\begin{center}
\begin{tabular}{||l|c|c|c|c|c|c|c||}
\hline
 $d$& $d=0$  & $d=1$  &$d=2$ &$d=3$ &$d=4$& $f^*$\\
\hline
\hline
$\lambda_d$& 0 & -1&-2.662 & -3.22&-4&-5\\
\hline 
\end{tabular}
\end{center}
\caption{MAXCUT: $n=11$; $Q(i,j)=1$ for all $i\neq j$.\label{table1}}
\end{table}
\begin{table}
\begin{center}
\begin{tabular}{||l|c|c|c|c|c|c|c||}
\hline
 $d$& $\lambda_0$  & $\lambda_1$  &$\lambda_2$ &$\lambda_3$ &$\lambda_4$& $f^*$\\
\hline
\hline
Ex1& 0 & -1.928 &-3.748&-5.22 &-6.37 &-7.946\\
\hline
Ex2& 0 & -1.56 &-3.103&-4.314 &-5.282 &-6.863\\
\hline
Ex3& 0 & -1.910 &-3.694&-5.078 &-6.161 &-8.032\\
\hline
Ex4& 0 & -2.164 &-4.1664&-5.7971&-7.06 &-9.198\\
\hline
Ex5& 0 & -1.825 &-3.560&-4.945 &-5.924 &-7.467\\
\hline
\end{tabular}
\end{center}
\caption{MAXCUT: $n=11$; $Q$ random.\label{table2}}
\end{table}

}\end{ex}

The above examples seem to indicate
that even though one chooses
a measure $\mu$ uniformly distributed on $\K$, 
one obtains a rapid decrease in the first iterations and then a slow convergence close to $f^*$. 
If on the one hand the convergence to $f^*$ is likely to be slow, 
on the other hand, one has to solve semidefinite programs 
(\ref{primal}) with only one variable! In fact solving the semidefinite program
(\ref{eigenvalue}) is computing the smallest {\it generalized eigenvalue} associated
with the pair of real symmetric matrices $(\M_d(f\,\y),\M_d(\y))$,
for which specialized codes are available (instead of using a solver for semidefinite programs).
However, one has to remember that the choice is limited to measures $\mu$ with $\supmu=\K$ and whose moments are available or easy to compute. Hence, the present methodology is so far limited to
simple sets $\K$ as described before. Finally, analyzing how the convergence to $f^*$ depends on $\mu$
is beyond the scope of the present paper and is a topic of further research. 

\subsection{Discusssion}

In nonlinear programming, sequences of upper bounds on the global minimum $f^*$
are usually obtained from feasible points $\x\in\K$, e.g., via some (local) minimization algorithm.
But for non convex problems, providing a sequence of upper bounds that converges to the global minimum $f^*$
is in general impossible unless one computes points on a grid whose mesh size tends to zero. In the above methodology one provides a monotone nonincreasing sequence of upper bounds 
converging to $f^*$ for polynomial optimization problems on sets $\K$, non necessarily compact but such that one may 
compute all moments of some finite Borel measure $\mu$ with $\supmu=\K$. In fact,
if there are only finitely many (say up to order $2d$) moments available then one
obtains a finite sequence of upper bounds. 

In contrast to the 
hierarchy of semidefinite relaxations in e.g. \cite{lasserresiopt,lasserrebook} which provide 
lower bounds converging to $f^*$ when $\K$ is a compact
basic semi-algebraic set, the convergence of the upper bounds to $f^*$ is only asymptotic and never finite, except
when $\K$ is a discrete set. However,
and even if we expect the convergence to be rather slow when close to $f^*$,
to our knowledge it is the first approach of this kind, and in a few iterations one may obtain
upper bounds which (even if crude) complements the lower bounds obtained 
in \cite{lasserresiopt} (in the compact case).

Also note that to solve (\ref{eigenvalue}) several improvements 
are possible. For instance,
we have already mentioned that it could be solved via specialized packages for generalized eigenvalue problems.
Next, if instead of using the canonical basis of monomial $(\x^\alpha)$, one now expresses the moment matrix $\M_d(\y)$ with rows and columns indexed in the basis of polynomials $(p_\alpha)\subset\R[\x]$ (up to degree $d$) {\it orthogonal}\footnote{A family of univariate polynomials $(p_k)\subset\R[x]$ is orthogonal with respect to a finite measure $\mu$ on $\R$ if
$\int p_ip_kd\mu=\delta_{i=k}$. For extensions to the multivariate case see e.g. \cite{dunkl,helton}.} with respect to $\mu$,
then $\M_d(\y)$ becomes the identity matrix. And so problem (\ref{eigenvalue}) reduces to a standard 
eigenvalue problem, namely that of computing the smallest eigenvalue of the (real and symmetric) localizing matrix 
$\M_d(f\,\y)$ (expressed in the basis of orthogonal polynomials)! And it turns out that computing the 
orthogonal polynomials is easy once the moment matrix $\M_d(\y)$ is available, since they
can be obtained via computing certain determinants, as explained in e.g. \cite{dunkl,helton}.

\subsection*{Inverse problem from moments} Finally, observe that the above methodology perfectly fits 
{\it inverse problems} from moments, where precisely
some Borel measure $\mu$ is known only from 
its moments (via some measurement device), and one wishes to recover (or approximately recover) its support $\K$ from the known moments; see e.g. the work of Cuyt et al. \cite{cuyt}
and the many references therein.
Hence if $f\in\R[\x]$
is fixed then by definition $f-f^*\geq0$ (on $\K$) provides a strong {\it valid} (polynomial) {\it inequality} for 
the unknown set $\K$. So 
computing an optimal solution $\lambda_d$ of (\ref{primal}) for $d$ sufficiently large, will provide an almost-valid
polynomial inequality 
$f-\lambda_d\geq0$ for $\K$.
One may even let $f\in\R[\x]_d$ be unknown and search
for the ``best" valid inequality $f(\x)-f^*\geq0$ where $f$ varies in some family (e.g. linear or quadratic polynomials) and minimizes ome appropriate (linear or convex) objective function of its vector of coefficents $\f$.

\section{Conclusion}

In this paper we have presented a new characterization of nonnegativity on a closed set $\K$ which is based on 
the knowledge of a single finite Borel measure $\mu$ with $\supmu=\K$. It permits to
obtain a hierarchy of spectrahedra which provides a nested sequence of outer approximations of the convex cone of polynomials of degree at most $d$, nonnegative on $\K$. When used in 
polynomial optimization for certain ``simple sets" $\K$, one obtains a hierarchy of semidefinite approximations 
(with only one variable) which provides a nonincreasing sequence of upper bounds converging to the global optimum, hence a complement to the sequence of upper bounds provided by the hierarchy of semidefinite relaxations defined in e.g. \cite{lasserresiopt,lasserrebook} when $\K$ is compact and basic semi-algebraic. A topic of further 
investigation is to analyze the efficiency of such an approach
on a sample of optimization problems on simple closed sets 
like the whole space $\R^n$, the positive orthant $\R^n_+$, a box, a simplex, or an ellipsoid, as well as for some inverse problems from moments.

\section{appendix}
\label{appendix}
\subsection*{Proof of Theorem \ref{noncompact}}
\begin{proof}
The {\it only if} part is exactly the same as in the proof of Theorem \ref{thmradon}.
For the {\it if} part, let 
$\z=(z_\alpha)$ be 
the sequence defined in (\ref{defnu}).
The sequence $\z$  is well defined because $\x\mapsto\x^\alpha f(\x)$ is $\mu$-integrable for all $\alpha\in\N^n$.
Indeed, if $f$ is a polynomial (so that $\mu$ is defined in (\ref{defmu})) we have seen that
all moments of $\mu$ are finite and since $\x^\alpha f(\x)$ is a polynomial
the  result follows. If $f$ is not a polynomial (so that $\mu$ is defined in (\ref{noncompact-1}))
then 
\[\int_\K\vert \x^\alpha\,f(\x)\vert\,d\mu(\x)\,\leq\,\int_\K \vert \x^\alpha\vert \exp(-\sum_i\vert x_i\vert)\,d\varphi(\x)
\leq\varphi(\R^n)\prod_{i=1}^n\alpha_i\l,\]
where we have used that $\vert x_i^{\alpha_i}\vert\leq \alpha_i\l \exp{\vert x_i\vert}$, and $\vert f\vert/(1+f^2)\leq 1$ for all $\x$.
As in the proof of Theorem \ref{thmradon}, the set function 
$B\mapsto \nu(B):=\int_B fd\mu$, $B\in\mathcal{B}$, 
is a signed Borel measure
because again $\nu$ can be written as the difference $\nu_1-\nu_2$ of the two
positive Borel measures $\nu_1,\nu_2$ in (\ref{decomp})-(\ref{hahn}).
With same majorizations as above, both $\nu_1$ and $\nu_2$ are finite Borel measures
and so $\nu$ is a finite signed Borel measure.

The same arguments as in the proof of Theorem \ref{thmradon} show that
$\M_d(\z)\succeq0$ for every $d\in\N$. Next, 
the sequence $\z$ satisfies the generalized Carleman's condition (\ref{carleman}).

Indeed, first consider the case where $f$ is a polynomial (and so $\mu$ is as in (\ref{defmu})).
Let $1\leq i\leq n$ be fixed arbitrary, and let $2s\geq{\rm deg}f$. Observe that
whenever $\vert \alpha\vert\leq k$, 
$\vert\x\vert^\alpha\leq \vert\x_j\vert^k$ on the subset $W_j:=\{\x\in \R^n\setminus [-1,1]^n\,:\,\vert x_j\vert=\max_i\vert x_i\vert\}$. And so, $\vert f(\x)\vert\leq \Vert f\Vert_1\, \vert\x_j\vert^{2s}$ for all $\x\in W_j$ (and where $\Vert f\Vert_1:=\sum_\alpha \vert f_\alpha\vert$). Hence,
\begin{eqnarray}
\nonumber
L_{\z}(x_i^{2k})&=&\int_\K f(\x)\,x_i^{2k}d\mu(\x)\\
\nonumber
&\leq&\int_{\K\cap [-1,1]^n} \vert f(\x)\vert\,x_i^{2k}d\mu(\x)+
\Vert f\Vert_1\,\sum_{j=1}^n\int_{\K\cap W_J} \,x_j^{2(k+s)}d\mu(\x)\\
\label{aux-carl1}
&\leq& \Vert f\Vert_1\mu(\K)+Mn\Vert f\Vert_1\,(2(k+s))\l\,\leq\,2Mn\Vert f\Vert_1\,(2(k+s))\l,
\end{eqnarray}
where $M$ is as in (\ref{bigM}) and assuming with no loss of generality that $\mu(\K)\leq Mn(2(k+s))\l$ (otherwise rescale $\varphi$).
And so we have
\begin{eqnarray*}
L_{\z}(x_i^{2k})^{-1/2k}&\geq& (2Mn\Vert f\Vert_1)^{-1/2k}\,\left((2(k+s))\l)^{-1/2(k+s)}\right)^{(k+s)/k}\\
&\geq& \frac{1}{2}\left((2(k+s))\l)^{-1/2(k+s)}\right)^{(k+s)/k}\\
&\geq&\frac{1}{2}\left(\frac{1}{2(k+s)}\right)^{(k+s)/k},
\end{eqnarray*}
where $k\geq k_0$ is sufficiently large so that $(2Mn\Vert f\Vert_1)^{-1/2k}\geq 1/2$. 
Therefore,
\[\sum_{k=1}^\infty
L_{\z}(x_i^{2k})^{-1/2k}\geq \frac{1}{2}\sum_{k=k_0}^\infty\left(\frac{1}{2(k+s)}\right)^{(k+s)/k}\,=\,+\infty,\]
where the last equality follows from $\sum_{k=1}^\infty (2k)^{-1}=+\infty$. Indeed,
$(\frac{1}{2(k+s)})^{(k+s)/k}=(\frac{1}{2(k+s)})(\frac{1}{2(k+s)})^{s/k}$ and
$(\frac{1}{2(k+s)})^{s/k}\geq 1/2$ whenever $k$ is sufficiently large, say $k\geq k_1$. Hence the sequence
$\z$ satisfies Carleman's condition (\ref{carleman}).

If $f$ is not a polynomial then $\mu$ is as in (\ref{noncompact-1}) and so
\begin{eqnarray}
\nonumber
L_\z(x_i^{2k})&=&\int x_i^{2k}\frac{\exp\left(-\sum_{i=1}^n\vert x_i\vert\right)}{1+f^2}\,f(\x) \,d\varphi(\x)\\
\nonumber
&\leq&\int x_i^{2k}\frac{\exp\left(-\sum_{i=1}^n\vert x_i\vert\right)}{1+f^2}\,\vert f(\x)\vert \,d\varphi(\x)\\
\label{aux-carl2}
&\leq&\int x_i^{2k}\,\exp\left(-\sum_{i=1}^n\vert x_i\vert\right)\,d\varphi(\x)\,\leq\,(2k)\l M,\end{eqnarray}
where we have used that $\vert f\vert/(1+f^2)\leq 1$ for all $\x$,
and $x_i^{2k}\leq (2k)\l \exp{\vert x_i\vert}$. And so again,
the sequence $\z$ satisfies Carleman's condition (\ref{carleman}).

Next, as $\M_d(\z)\succeq0$ for every $d\in\N$, by Proposition
\ref{prop-berg}, $\z$ is the moment sequence of a measure $\psi$ on $\R^n$ and $\psi$ is determinate. In addition, from the definition
(\ref{defnu}) of $\nu$ and $\z$, we have
\begin{equation}
\label{eq1}
\int_{\R^n} \x^\alpha\,d\psi(\x) \,=\,z_\alpha\,=\,\int_\K \x^\alpha d\nu(\x),\qquad\forall\,\alpha\in\N^n.\end{equation}
But then using $\nu=\nu_1-\nu_2$ in (\ref{decomp})-(\ref{hahn}), (\ref{eq1}) reads
\begin{equation}
\label{eq2}
\int_{\R^n} \x^\alpha\,d(\psi+\nu_2)(\x) \,=\,\int_\K \x^\alpha d\nu_1(\x),\qquad\forall\alpha\in\N^n.\end{equation}
Let $\v=(v_\alpha)$, $\alpha\in\N^n$, 
be the sequence of moments associated with 
$\nu_1$. Of course, $\M_d(\v)\succeq0$ for all $d\in\N$.
Next, 
\[L_\v(x_i^{2k})\,=\,\int_{B_1}x_i^{2k}\,f(\x)\,d\mu(\x)\,\leq\,\int_{\K}x_i^{2k}\,\vert f(\x)\vert\,d\mu(\x),\]
and so, depending on whether $f$ is a polynomial or not, we obtain
$L_\v(x_i^{2k})\leq 2Mn\Vert f\Vert_1\,(2(k+s))\l$ as in (\ref{aux-carl1}) or
$L_\v(x_i^{2k})\leq (2k)\l M$ as in (\ref{aux-carl2}). In both cases the sequence $\v$ satisfies Carleman's condition (\ref{carleman})
and since $\M_d(\v)\succeq0$ for all $d\in\N$, by Proposition \ref{prop-berg},
$\nu_1$ is moment determinate.
But then (\ref{eq2}) yields $\nu_1=\psi+\nu_2$, or equivalently, $\psi=\nu_1-\nu_2\,(=\nu)$, that is,
$\nu$ is a positive measure. 
The rest of the proof is exactly the same as
for proof of Theorem \ref{thmradon}.
\end{proof}

\end{document}